\documentclass[letterpaper,10pt]{amsart}
\usepackage{amsmath,amssymb,amsxtra,amsthm}
\usepackage{graphicx}

\theoremstyle{plain}
\newtheorem{theorem}{Theorem}[section]
\newtheorem{corollary}[theorem]{Corollary}
\newtheorem{lemma}[theorem]{Lemma}
\newtheorem{proposition}[theorem]{Proposition}

\newtheorem{conjecture}[theorem]{Conjecture}

\theoremstyle{definition}

\newtheorem{remark}[theorem]{Remark}

\newcommand{\B}{\mathbb}
\newcommand{\C}{\mathcal}
\newcommand{\F}{\mathfrak}

\newcommand{\ga}{\alpha}
\newcommand{\gb}{\beta}

\newcommand{\eps}{\varepsilon}
\newcommand{\gf}{\varphi}

\newcommand{\gl}{\lambda}

\newcommand{\gs}{\sigma}

\newcommand{\tbf}{\textbf}

\DeclareMathOperator{\Gal}{Gal}
\DeclareMathOperator{\spec}{Spec}

\begin{document}
\title[Transcendence of generating functions]{Transcendence of generating functions whose coefficients are multiplicative}
\author{Jason P.~Bell}
\author{Nils Bruin}
\author{Michael Coons}%
\address{Simon Fraser University, 8888 University Drive, Burnaby, British Columbia, V5A 1S6}
\email{jpb@sfu.ca, nbruin@sfu.ca}
\address{University of Waterloo, Dept.~of Pure Mathematics, Waterloo, Ontario, N2L 3G1}
\email{mcoons@math.uwaterloo.ca}
\thanks{The research of J.~P.~Bell and N.~Bruin is supported in part by a grant from NSERC of Canada, and the research of M.~Coons is supported by a Fields--Ontario Fellowship.}%
\subjclass[2000]{Primary 11N64; 11J91 Secondary 11B85}%
\keywords{algebraic functions, multiplicative functions, automatic sequences}%
\date{March 10, 2010}

\begin{abstract} In this paper, we give a new proof and an extension of the following result of B\'ezivin. Let $f:\B{N}\to K$ be a multiplicative
function taking values in a field $K$ of characteristic $0$ and write
$F(z)=\sum_{n\geq 1} f(n)z^n\in K[[z]]$ for its generating series.
Suppose that $F(z)$ is  algebraic over $K(z)$. Then either there is a natural number $k$ and a periodic multiplicative function $\chi(n)$ such that $f(n)=n^k \chi(n)$ for all $n$, or $f(n)$ is eventually zero.  In particular, 
$F(z)$ is either transcendental or rational.
For $K=\B{C}$, we also prove that if $F(z)$ is a $D$-finite generating series of a multiplicative function, then $F(z)$ is either transcendental or rational.
\end{abstract}

\maketitle


\section{Introduction}

In 1906, Fatou \cite{Fat1} investigated algebraic power series with integer coefficients. A power series with integer coefficients is either a polynomial or has radius of convergence of at most one. It is therefore natural to consider the special class of power series having integer coefficients which converge inside the unit disk. Fatou proved the following result.

\begin{theorem}[Fatou \cite{Fat1}]\label{fatou} If $F(z)=\sum_{n\geq 1}f(n)z^n\in\B{Z}[[z]]$ converges inside the unit disk, then either $F(z)\in\B{Q}(z)$ or  $F(z)$ is transcendental over $\B{Q}(z)$. Moreover, if $F(z)$ is rational, then each pole is located at a root of unity.
\end{theorem}

Carlson \cite{Car1}, proving a conjecture of P\'olya, added to Fatou's theorem.

\begin{theorem}[Carlson \cite{Car1}] A series $F(z)=\sum_{n\geq 1}f(n)z^n\in\mathbb{Z}[[z]]$ that converges inside the unit disk is either rational or it admits the unit circle as a natural boundary.
\end{theorem}

Recall that if $f(n)=O(n^d)$ for some $d$, the series $F(z)=\sum_{n\geq 1}f(n)z^n$ $\in\mathbb{Z}[[z]]$ has the unit circle as a natural boundary. By the two theorems above, such a series is either rational or transcendental over $\mathbb{Q}(z)$. This gives very quick transcendence results over $\B{Q}(z)$ for series $F(z)$ with $f(n)$ equal to any of the number--theoretic functions $\varphi(n), \tau(n^2), \tau^2(n), \omega(n),$ or $\Omega(n)$.
Here we follow the usual notation, where $\varphi(n)$ is the Euler totient function, $\tau(n)$ is the number of divisors of $n$, $\omega(n)$ is the number of distinct prime divisors of $n$, and $\Omega(n)$ is the number of prime divisors of $n$ counting multiplicity. 

Results like these are widely known. Indeed, Banks, Luca, and Shparlinski \cite{banks} have shown that $\sum_{n\geq 1} f(n)z^n\in\B{Z}[[z]]$ is irrational for $f(n)$ any one of $\varphi(n), \tau(n),$ $ \sigma(n),$  $\lambda(n), \mu(n), \omega(n), \Omega(n), p(n),$ or $\rho(n)$, where 
we write
$\lambda(n)=(-1)^{\Omega(n)}$ for Liouville's function, $\mu(n)$ for  the M\"obius function, and $\rho(n)=2^{\omega(n)}$ for the number of squarefree divisors of $n$. Transcendence results for some of these functions were given previously by Allouche \cite{All3, All1} and Yazdani \cite{Yaz1}. Taking into account the results of Borwein and Coons \cite{BC1} and Coons \cite{Coons1} completes the picture. Indeed, Coons \cite{Coons1} took this further by proving transcendence over $\B{F}_p(z)$ for many of these functions, after one considers a reduction mod a prime $p$.

All of the aforementioned results suggest that generating series of multiplicative functions are either rational or transcendental, but never algebraic.
S\'ark\"{o}zy \cite{Sar1} characterises multiplicative functions with rational generating series.
He shows that if $f$ is a multiplicative function from $\B{N}$ to $\B{C}$ such that $\sum_{n\geq 1}f(n)z^n$ is rational over $\B{C}(z)$ then
either $f$ is eventually zero, or there is a natural number $k$ and a periodic multiplicative function $\chi$ such that $f(n)=n^k \chi(n)$ for all $n$.

B\'ezivin \cite{Bez1} extended S\'ark\"{o}zy's results to algebraic power series as well as to a larger subset of $D$--finite power series.

\begin{theorem}[B\'ezivin \cite{Bez1}]\label{Bezivin} Let $f:\B{N}\to \B{C}$ be a multiplicative function, and suppose that its generating series $F(z)=\sum_{n\geq 1} f(n)z^n\in \B{C}[[z]]$ is algebraic over $\B{C}(z)$. Then either there is a natural number $k$ and a periodic multiplicative function $\chi:\B{N}\to \B{C}$ such that $f(n)=n^k \chi(n)$ for all $n$, or $f(n)$ is eventually zero.
\end{theorem}

\begin{theorem}[B\'ezivin \cite{Bez1}]\label{BezivinDF} Let $f:\B{N}\to \B{C}$ be a multiplicative function with image contained in $\B{R}$ or $\B{C}^*$, and suppose that its generating series $F(z)=\sum_{n\geq 1} f(n)z^n$ is $D$--finite. Then either there is an integer $k$ and a periodic multiplicative function $\chi:\B{N}\to \B{C}$ such that $f(n)=n^k \chi(n)$ for all $n$, or $f(n)$ is eventually zero.
\end{theorem}

In this paper, we provide a new more number-theoretic proof, and slight extension, of Theorem \ref{Bezivin} as well as the extension to the complete $D$--finite case. More specifically, for algebraic generating series
 we extend the result to multiplicative functions $f$ taking values in any field of characteristic zero, and for $D$--finite generating series we remove the nonzero restriction for complex valued $f$.

\begin{theorem}\label{main} Let $K$ be a field of characteristic $0$, let $f:\B{N}\to K$ be a multiplicative function, and suppose that its generating series $F(z)=\sum_{n\geq 1} f(n)z^n\in K[[z]]$ is algebraic over $K(z)$. Then either there is a natural number $k$ and a periodic multiplicative function $\chi:\B{N}\to K$ such that $f(n)=n^k \chi(n)$ for all $n$, or $f(n)$ is eventually zero.
\end{theorem}

\begin{theorem}\label{mainD} Let $f:\B{N}\to \B{C}$ be a multiplicative function, and suppose that its generating series $F(z)=\sum_{n\geq 1} f(n)z^n\in \B{C}[[z]]$ is $D$--finite over $\B{C}(z)$. Then either there is an integer $k$ and a periodic multiplicative function $\chi:\B{N}\to \B{C}$ such that $f(n)=n^k \chi(n)$ for all $n$, or $f(n)$ is eventually zero.
\end{theorem}

These theorems can be viewed in a wider context. Many results from the literature show that a rational--transcendental dichotomy holds for power series (respectively real numbers) whose coefficients (respectively base $k$ expansions) satisfy a property that is independent of being algebraic.  For example, the results of Fatou \cite{Fat1}, Carlson \cite{Car1}, and Cobham \cite{Cob1} show that power series with integer coefficients that do not grow too fast are either transcendental or rational.  More recently, Adamczewski and Bugeaud \cite{AB} showed that a real number which is both algebraic and automatic is necessarily rational. Similarly, Theorems \ref{Bezivin} and \ref{main} show that a power series whose coefficients are multiplicative is either transcendental or rational

Our proof of Theorem \ref{main} proceeds by a sequence of reductions. In Section \ref{sec2} we prove Theorem \ref{main} for $\B{Z}$--valued $f$ using Theorem~\ref{fatou}. A vital element is a generalization of a theorem of Eisenstein concerning the coefficients of the power series expansion of an algebraic function. In Sections \ref{sec3} and \ref{sec4}, we use this generalization to prove Theorem \ref{main} in the case that $K=\B{Q}$ and in the case that $K$ is a finite extension of $\B{Q}$, respectively. To complete the proof, in Section \ref{sec5}, we use the Lefschetz principle and the Nullstellensatz to prove the general case.  Finally, in Section \ref{Df} we prove Theorem \ref{mainD} which completes our characterization. 

We note that throughout our considerations, we use the fact that given a field $K$ and an extension $L$ of $K$, a power series $F(z)\in K[[z]]$ is algebraic over $K(z)$ if and only if it is algebraic over $L(z)$.

Also, in many places in this paper, we use the fact that a $D$--finite power series $F(z):=\sum_{n\geq 1}f(n)z^n$ has coefficients which are $P$--recursive \cite[Theorem 6.4.3]{Stan1}, which implies that the coefficients are contained in a finitely generated extension of $\B{Q}$.
This fact is easily established in the following way.
Suppose that $f(n)$ is $P$--recursive with polynomial recurrence $$\sum_{i=0}^d P_i(n) f(n-i)=0.$$
We let $R$ be the $\B{Q}$--algebra generated by $f(1),\ldots,f(N)$ and the coefficients of $P_0, ..., P_d$ where $N>d$ is chosen so that $P_0(n)$ is nonzero for $n>N.$ Then $f(n)$ takes values in the localization $S^{-1}R$, where $S =\{P_0(N+1),P_0(N+2),\ldots\}.$

\section{$\B{Z}$--valued multiplicative functions}\label{sec2}

In this section, we show that the generating function of a $\B{Z}$--valued multiplicative function $f$ is either rational or transcendental. To do this, we need some estimates on the growth of $f$. This is done via a generalization of a theorem of Eisenstein \cite[Problem VIII.153]{PolSze1}, who proved the following lemma in the case that $R=\B{Z}$ and $K=\B{Q}$.

\begin{lemma}\label{geneis} Let $R$ be an integral domain with field of fractions $K$. If $F(z)=\sum_{n\geq 1}f(n)z^n\in K[[z]]$ is the power series expansion of an algebraic function over $K(z)$, then there is a function $g:\B{N}\to R$ and a nonzero element $c\in R$ such that $f(n)=g(n)/c^n$ for every $n$.
\end{lemma}

\begin{proof}
Pick a polynomial relation
$$P_r(z)F(z)^r+\cdots + P_0(z)\ = \ 0$$ with 
$P_0(z),\ldots ,P_r(z)\in R[z]$, not all zero, and with $r$ minimal.
For $1\leq j \leq r$, let 
\begin{equation}
F_j(z) := \sum_{i=j}^r {i\choose j} P_i(z)F(z)^{i-j}.
\end{equation}
By the minimality of $r$ the power series $F_1(z),\ldots , F_r(z)\in K[[z]]$ are nonzero and hence there exist natural numbers $n_1,\ldots , n_r$ where $n_j$ is the smallest nonnegative integer such that the coefficient of $z^{n_j}$ in $F_j(z)$ is nonzero. 

Let $$Q_n(z):=f(1)z+f(2)z^2+\cdots+f(n)z^n.$$ Then for $1\leq j\leq r$ $$\lim_{n\to\infty}\sum_{i=j}^r { i\choose j} P_i(z)Q_n(z)^{i-j}\to F_j(z),$$ where convergence is in the $(z)$--adic topology on $K[[z]]$. Hence there is a natural number $m>\max (n_1,\ldots, n_r)$ such that the coefficient of $z^{n_j}$ in the polynomial $$T_j(z):=\sum_{i=j}^r { i\choose j} P_i(z)Q_m(z)^{i-j}$$ is nonzero.
Define $$G(z):=\frac{F(z)-Q_m(z)}{z^{m}}=\sum_{n\geq 1}f(m+n)z^n.$$ Then $$F(z)=z^{m}G(z)+Q_m(z),$$
so substituting this expression into our polynomial relation we find
\begin{align*} 0&=\sum_{i=0}^r P_i(z)\left[z^{m}G(z)+Q_m(z)\right]^i\\
&=\sum_{i=0}^r P_i(z)\sum_{j=0}^i {i\choose j} z^{mj}G(z)^jQ_m(z)^{i-j}\\
&=\sum_{j=0}^r G(z)^j\left[\sum_{i=j}^r {i\choose j} z^{mj}Q_m(z)^{i-j}P_i(z)\right]\\
&=\sum_{j=0}^r z^{mj}T_j(z)G(z)^j.
\end{align*}
Note that $z^{mj}T_j(z)$ has a zero of order $mj+n_j$ at $z=0$ for $1\leq j\leq r$. Since $m>\max(n_1,\ldots,n_r)$ we have $$mr+n_r>m(r-1)+n_{r-1}>\cdots>m+n_1.$$ Letting $S_j(z):= z^{m(j-1)-n_1}T_j(z),$ we see that
\begin{equation}\label{SG0}\sum_{j=0}^r S_j(z)G(z)^j=0,\end{equation}
that the constant coefficients of $S_2(z),\ldots, S_r(z)$ are zero and that the constant coefficient of $S_1(z)$ is nonzero. Moreover, $S_0(z)\in R[z,z^{-1}]$ since it was obtained by multiplying the polynomial $T_0(z)$ by $z^{-m-n_1}$. 
We also have that
$$S_0(z)=-\sum_{j=1}^r S_j(z)G(z)^j,$$
which is a power series in $K[[z]]$, so $S_0(z)$ is, in fact, in $R[z]$.

Let $a\in R$ denote the constant coefficient of $S_1(z)$. We claim that $f(n+m)a^{n}\in R$ for all $n\in\B{N}$. We show this by induction on $n$. Looking at the coefficient of $z$ in both sides of  \eqref{SG0} gives $$a\cdot f(m+1)+b=0,$$ where $b\in R$ is the coefficient of $z$ in $S_0(z)$. Hence $a\cdot f(m+1)\in R$, proving the claim for $n=1$. Now assume that the claim is true for all natural numbers less than $n$. We look at the coefficient of $z^n$ in \eqref{SG0}. This shows that $a\cdot f(m+n)=S_1(0)f(m+n)$ is an $R$--linear combination of products of the form $$f(m+i_1)f(m+i_2)\cdots f(m+i_d)$$ where $d\leq r$, and $i_1,\ldots,i_d\geq 1$, and $i_1+i_2+\cdots+i_d<n.$ By the inductive hypothesis, $f(m+i_1)a^{i_1},\ldots,f(m+i_d)a^{i_d}\in R$, and hence $$f(m+i_1)\cdots f(m+i_d)a^{n-1}\in R,$$ whenever $i_1+i_2+\cdots+i_d<n.$ And so, $S_1(0)f(m+n)a^{n-1}\in R$ and this is equal to $f(m+n)a^n$, proving the claim.

Pick a nonzero $b\in R$ such that $f(1)b,\ldots f(m)b\in R$, then $f(n)(ab)^n\in R$ for all natural numbers $n$. Taking $c=ab$ completes the proof.
\end{proof}

For convenience, we take $\B{C}_\infty:=\B{C}$ and $|\cdot|_\infty$ to be the Euclidean norm on $\B{C}$. Note that the following lemma shows that the generating function of a $\B{C}_p$--valued multiplicative function is either transcendental or it converges in the unit disk of $\B{C}_p$, where $p\in\{2,3,5,\ldots\}\cup\{\infty\}$.

\begin{lemma}\label{convlemma} Let $p\in\{2,3,5,\ldots\}\cup\{\infty\}$ and $F(z):=\sum_{n\geq 1}f(n)z^n\in\B{C}_p[[z]]$ be the power series expansion of an algebraic function. If $f(n)$ is a multiplicative function, then $$\limsup_{n\to\infty}|f(n)|_p^{\frac{1}{n}}\leq 1.$$
\end{lemma}

\begin{proof} Denote $\ga:=\limsup_{n\to\infty}|f(n)|_p^{\frac{1}{n}}.$ If $p=\infty$, then $\ga$ is finite by the Implicit Function Theorem \cite[Theorem 2.3.1]{Kra1}.
If $p\neq \infty$, we note that the closed unit ball in $\B{C}_p$ is a ring whose field of fractions is $\B{C}_p$. Lemma \ref{geneis} implies that $f(n)=\frac{g(n)}{A^n}$ for some $\B{C}_p$--valued sequence $(g(n))_{n\geq 1}$ with $|g(n)|_p\leq 1$ and some nonzero $A\in\B{C}_p$. Then $$|f(n)|_p=\left|\frac{g(n)}{A^n}\right|_p\leq |A^{-1}|_p^n$$ and so $$\limsup_{n\to\infty}|f(n)|_p^{\frac{1}{n}}\leq |A^{-1}|_p<\infty.$$ Thus $\ga<\infty$.

Towards a contradiction, suppose $\ga>1$. Then for any $\eps>0$ there is an $N>0$ such that for all $n>N$ \begin{equation}\label{peps}|f(n)|_p<(\ga+\eps)^n\end{equation} and also there are infinitely many $n>N$ such that \begin{equation}\label{meps}(\ga-\eps)^n<|f(n)|_p.\end{equation}

Suppose $n$ is not a prime power. We write $n=ab$ with $\gcd(a,b)=1$ and $a,b>1$, so that $$|f(n)|_p=|f(ab)|_p=|f(a)f(b)|_p<(\ga+\eps)^{a+b}.$$ Since $ab=n$, we have for $n\geq 12$ that $a+b\leq \frac{2}{3}n,$ which yields $$|f(n)|_p=|f(ab)|_p<(\ga+\eps)^{\frac{2}{3}n}.$$  

Now suppose $n=q^k$ for some $k\geq 1$ and some prime $q$. Since $F(z)$ is algebraic it is $D$--finite \cite[Theorem 6.4.6]{Stan1}. Hence by \cite[Proposition 6.4.3]{Stan1}, there exist $r+1$ polynomials $P_0,P_1,\ldots,P_r\in\B{C}_p[z]$, with $P_0(z)$ not identically zero, such that $$P_0(n)f(n)=P_1(n)f(n-1)+P_2(n)f(n-2)+\cdots+P_r(n)f(n-r).$$ Furthermore, we can assume that the coefficients of the $P_i(z)$ are contained in the closed unit ball of $\B{C}_p$.

If $p=\infty$, for all $n$ sufficiently large we have that $1\leq |P_0(n)|_p$. Furthermore there is a $k$ such that $|P_i(n)|_p\leq n^k$ for each $0\leq i\leq r$. On the other hand, if $p\neq\infty$, there is a $k$ such that for all $n$ sufficiently large we have $\frac{1}{n^k}\leq |P_0(n)|_p$ and  $|P_i(n)|_p\leq 1$ for each $0\leq i\leq r$. It follows that
\begin{align*}|f(n)|_p &\leq n^{k}\cdot\max\{|f(n-1)|_p,\ldots,|f(n-r)|_p\}.\end{align*} 

Define the sets $S$ and $T$ by $$S:=\{n: |f(n)|_p>(\ga-\eps)^n\},\quad\mbox{and}\quad T:=\{n:|f(n)|_p>(\ga+\eps)^{\frac{2}{3}n}\}.$$ There is a finite set $S_0$ such that
$$S=S_0\cup \{q^k: k\geq 1, q\mbox{ prime}\}.$$

For each $n\in S$ and $0\leq i_1,i_2,\ldots\leq d,$ we have $$|f(n-i_1)|_p>\frac{(\ga-\eps)^n}{n^{k}},$$
and $$|f(n-i_1-i_2)|_p>\frac{(\ga-\eps)^n}{n^k(n-i_1)^k}>\frac{(\ga-\eps)^n}{(n^{k})^2}.$$
and in general for any $l$ that
$$|f(n-i_1-i_2-\cdots-i_l)|_p\geq \frac{(\ga-\eps)^n}{(n^{k})^l}.$$
For large enough $n$ we have $\frac{(\ga-\eps)^n}{n^{kl}}>(\ga+\eps)^{\frac{2}{3}n}$ and hence $$|f(n-i_1-i_2-\cdots-i_l)|_p>(\ga+\eps)^{\frac{2}{3}n}.$$
If $\frac{(\ga-\eps)^n}{n^{kl}}>(\ga+\eps)^{\frac{2}{3}n}$, then $n-i_1-i_2-\cdots-i_l\in T$ and hence is a prime power.

Considering the inequality, we have \begin{align*}\frac{(\ga-\eps)^n}{n^{kl}}>(\ga+\eps)^{\frac{2}{3}n} &\Longleftrightarrow \left[\frac{\ga-\eps}{(\ga+\eps)^\frac{2}{3}}\right]^n>n^{kl}\\
&\Longleftrightarrow n\log\left(\frac{\ga-\eps}{(\ga+\eps)^\frac{2}{3}}\right) > lk\log n. \end{align*} Since $\ga>1$ this tells us that there is a $c_0>0$ so that
$$c_0n>lk\log n$$
and
$$l<\frac{c_0n\cdot k}{\log n}=\frac{\kappa}{r} \frac{n}{\log n}$$
where we define $\kappa$ by $c_0kr=\kappa$. This implies that in the interval $$\left[n-\frac{\kappa n}{\log n}, n\right]$$ there are $\frac{\kappa}{r} \frac{n}{\log n}$ prime powers. This contradicts the prime number theorem.
\end{proof}

The following lemma is used in the characterization of multiplicative functions with rational generating functions. 

\begin{lemma}\label{declension} Suppose that $f:\B{N}\to\B{C}$ is a multiplicative function which for large enough $n$ is given by $f(n)=n^k\chi(n)$, where $\chi:\B{N}\to\B{C}$ is a periodic multiplicative function (not identically 0). Then $f(n)=n^k\chi(n)$ for all $n\in\B{N}$.
\end{lemma}

\begin{proof} Suppose not. Then there is a maximal $n_0$ for which $f(n)\neq n^k\chi(n)$. Note that there are infinitely many primes $p$ for which $\chi(p)\neq 0$, otherwise the set of natural numbers $n$ for which $\chi(n)\neq 0$ would have zero density, and this is not possible for a periodic function that is not identically zero. Thus we can pick a prime $p>n_0$ such that $f(p)=p^k\chi(p)\neq 0$. Then $$f(n_0)=\frac{f(pn_0)}{f(p)}=\frac{(pn_0)^k\chi(pn_0)}{p^k\chi(p)}=n_0^k\chi(n_0),$$ which is a contradiction.
\end{proof}

To give the main result of this section, we will use Theorem \ref{fatou} of Fatou and the following result of S\'ark\"ozy.

\begin{theorem}[S\'ark\"ozy \cite{Sar1}]\label{ratC} Let $f:\B{N}\to\B{C}$ be a multiplicative function and suppose that $F(z)=\sum_{n\geq 1}f(n)z^n\in\B{C}[[z]]$ is rational over $\B{C}(z)$. Then either $f(n)$ is eventually zero or there is a nonnegative integer $k$ and a periodic multiplicative function $\chi$ such that $f(n)=n^k\chi(n).$
\end{theorem}

\begin{corollary}\label{AlgZ} Let $f:\B{N}\to\B{Z}$ be a multiplicative function and suppose that $F(z)=\sum_{n\geq 1}f(n)z^n\in\B{Z}[[z]]$ is the power series expansion of a function that is algebraic over $\B{Q}(z)$. Then either $f(n)$ is eventually zero or there is a nonnegative integer $k$ and a periodic multiplicative function $\chi$ such that $f(n)=n^k\chi(n).$
\end{corollary}

\begin{proof} By Lemma \ref{convlemma}, using the Euclidean norm, $F(z)=\sum_{n\geq 1}f(n)z^n\in\B{Z}[[z]]$ is analytic in the open unit disk. Using Theorem \ref{fatou}, we see that $F(z)$ is rational. The result follows from Theorem \ref{ratC}.
\end{proof}

\section{$\B{Q}$--valued multiplicative functions}\label{sec3}

In this section, we consider the case $K=\B{Q}$ of Theorem \ref{main} via a reduction to the $\B{Z}$--valued case handled in Section \ref{sec2}.

\begin{lemma}\label{pfat} Let $F(z):=\sum_{n\geq 1} f(n) z^n\in\B{Q}[[z]]$ be the power series expansion of an algebraic function over $\B{Q}(z)$ and $p$ be finite. If $F(z)$ converges in $\B{C}_p$ for $|z|_p<1$, then $|f(n)|_p$ is uniformly bounded.
\end{lemma}

\begin{proof} Since $F(z)$ is algebraic, we have $P(z,F(z))=0$ for some polynomial $$P(x,y)=\sum_{j=0}^k P_j(x)y^j\in\B{Z}[x,y].$$ Let $\{\ga_1,\ldots,\ga_m\}$ denote the set of zeros of $P_0(x)\cdots P_k(x)$. We can pick $\eps>0$ 
such that the subset $X$ of $\B{C}_p$ defined by $$X:=\overline{B_{\B{C}_p}(0,1)}\backslash \bigcup_{j=1}^m B_{\B{C}_p}(\ga_j,\eps),$$ has the property that $\{|x|_p:x\in X\}\supset \left(\frac{1}{2},1\right).$

The polynomials $P_0(x),\ldots,P_k(x)$ are uniformly bounded on $X$ because $X$ itself is bounded. Since every $x\in X$ has $|x|_p\leq 1$, we see that there is a positive constant $c_1$ such that for all $x\in X$ we have  \begin{equation}\label{ubP}|P_0(x)|_p,\ldots,|P_k(x)|_p\leq c_1.\end{equation} Furthermore, each $P_i(x)$ can be written in the form $P_i(x)=a(x-\gb_1)\cdots (x-\gb_d),$ with $\{\gb_1,\ldots,\gb_d\}\subseteq \{\ga_1,\ldots,\ga_m\}$. Moreover, if $x\in X$ then $|x-\ga_j|_p\geq \eps$, whence there is a positive constant $c_2$ such that \begin{equation}\label{lbP}|P_0(x)|_p,\ldots,|P_k(x)|_p\geq c_2.\end{equation}  Let $$S:= \bigcup_{\substack{\{n:f(n)\neq 0\}\\ m\neq n}}\left\{r:r^{n-m}=|f(m)|_p/|f(n)|_p\right\}.$$ In particular, $S$ is countable.

Thus we can find a sequence $\{\gl_j\}\subseteq X$ with $|\gl_j|_p\notin S$ such that $|\gl_j|_p\uparrow 1$. If $|\gl|_p<1$ and $|\gl|_p\notin S$ then all nonzero terms of the sequence $\{|f(n)\gl^n|_p\}$ are distinct, and $$|f(n)\gl^n|_p\to 0.$$ Hence \begin{equation}\label{star}|F(\gl)|_p=\left|\sum_n f(n)\gl^n\right|_p=\max_n\left\{|f(n)\gl^n|_p\right\}.\end{equation} 

Also, if $\gl\in X$, then $$c_2\leq|P_0(\gl)|_p,\ldots,|P_k(\gl)|_p\leq c_1$$ and so \begin{align*} |P_k(\gl)F(\gl)^k|_p &= |P_{k-1}(\gl)F(\gl)^{k-1}+\cdots+P_0(\gl)|_p\\
&\leq \max_{0\leq i\leq k-1}|P_i(\gl)|_p|F(\gl)|_p^i\\
&\leq c_1\cdot\max (|F(\gl)|_p^{k-1},1).\end{align*} But $|P_k(\gl)|_p\geq c_2,$ and so $|F(\gl)|_p^k\leq c_1c_2^{-1}\max(|F(\gl)|_p^{k-1},1)$. Thus \begin{equation}\label{dagger}|F(\gl)|_p\leq \max(c_1c_2^{-1},1)\end{equation} for $\gl\in X$. 

Hence if $|\gl|_p\notin S$, we have $|f(n)|_p|\gl|_p^n\leq \max(c_1c_2^{-1},1)$ by combining \eqref{star} and \eqref{dagger}. In particular, using our sequence $\{\gl_j\}$ we see $|f(n)|_p|\gl_j|_p^n\leq \max(c_1c_2^{-1},1)$ for all $j$. Since $|\gl_j|_p\to 1$, we have $|f(n)|_p\leq \max(c_1c_2^{-1},1)$ and so the $|f(n)|_p$ are uniformly bounded.
\end{proof}

\begin{theorem} Let $f:\B{N}\to\B{Q}$ be a multiplicative function, $F(z):=\sum_{n\geq 1} f(n) z^n\in\B{Q}[[z]]$, and suppose that $F(z)$ is algebraic over ${\B{Q}(z)}$. Then either $f(n)$ is eventually zero or there is a nonnegative integer $k$ and a periodic multiplicative function $\chi$ such that $f(n)=n^k\chi(n).$
\end{theorem}

\begin{proof} There is a $P(z,y)\in\B{Z}[z,y]$ such that $P(z,F(z))=0$. By Lemma \ref{geneis} there is a natural number $A$ such that $$F(z)=\sum_{n\geq 1}\frac{g(n)}{A^n}z^n$$ where $g(n)\in\B{Z}$ for each $n$. 

Let $p$ be a prime. We consider $F(z)\in\B{Q}[[z]]\subset\B{C}_p[[z]]$. By Lemma \ref{convlemma} $$\limsup_{n\to\infty} \left|\frac{g(n)}{A^n}\right|_p^{\frac{1}{n}}\leq 1.$$ Notice that if $p\nmid A$, then $|g(n)/A^n|_p\leq 1$, and if $p|A$ then by Lemma \ref{pfat}, we have that $|g(n)/A^n|_p$ is uniformly bounded. Consequently there exists a natural number $N$ such that $N\cdot F(z)$ is a power series with integer coefficients; moreover, $N\cdot F(z)$ is algebraic. By using Lemma \ref{convlemma} with norm $|\cdot|_\infty$, we see that $N\cdot F(z)$ is analytic in the open unit disk of $\B{C}$. Theorem \ref{fatou} implies that $F(z)$ is rational. The result now follows from Theorem \ref{ratC}.
\end{proof}

\section{Finite extensions of $\B{Q}$}\label{sec4}

Taking the results of the previous sections further, we now prove Theorem \ref{main} in the case that $K$ is a finite extension of $\B{Q}$.

\begin{theorem}\label{finext} Let $K$ be a finite extension of $\B{Q}$ and $f:\B{N}\to K$ be a multiplicative function. If $F(z):=\sum_{n\geq 1}f(n)z^n\in K[[z]]$ is the power series expansion of an algebraic function over $K(z)$. Then either $f(n)$ is eventually zero or there is a nonnegative integer $k$ and a periodic multiplicative function $\chi:\B{N}\to K$ such that $f(n)=n^k\chi(n).$
\end{theorem}

\begin{proof} Suppose not. Without loss of generality, we may assume that $K$ is a Galois extension of $\B{Q}$. Pick a $\B{Q}$--basis $\{\ga_1,\ldots,\ga_r\}$ for $K$. In addition, we can assume this basis is {\em normal}, meaning that $\Gal(K/\B{Q})$ acts on the $\ga_i$ by permutation.

Then there exist $\B{Q}$--valued functions $f_1(n),\ldots,f_r(n)$ such that $$f(n)=f_1(n)\ga_1+f_2(n)\ga_2+\cdots+f_r(n)\ga_r.$$ Each of $F_i(z):=\sum_{n\geq 1}f_i(n)z^n\in\B{Q}[[z]]$ is algebraic over $\B{Q}(z)$. To see this, note that Theorem 1 of \cite{Gow1} gives that every $\B{Q}$--linear endomorphism of $K$ can be expressed as a $K$--linear combination of elements of $\Gal(K/\B{Q})$.

For $1\leq i\leq r$, let $M_i:K\to K$ be a $\B{Q}$--linear map sending $\ga_i\mapsto 1$ and $\ga_j\mapsto 0$ for $j\neq i$. Then $M_i(f(n))=f_i(n)$. On the other hand, $M_i$ is a $K$--linear combination of elements of $\Gal(K/\B{Q})$, and since applying automorphisms preserves algebraicity, and algebraic power series are closed under taking linear combinations, we see that $F_i(z)=M_i(F(z))$ is algebraic.

By Lemma \ref{geneis}, there exists a natural number $A$ and maps $g_1,\ldots,g_r:\B{N}\to\B{Z}$ such that $f_i(n)=g_i(n)/A^n.$ Let $p$ be a prime dividing $A$. We identify $K$ with its image under a field embedding into $\B{C}_p$. Then $|f_i(n)|_p=|M_i(f(n))|_p$; moreover, $M_i(f(n))$ is a $K$--linear combination of multiplicative functions (these are the functions $f^\gs(n)$ where $\gs\in\Gal(K/\B{Q})$) whose generating functions in $z$ are algebraic over $\B{Q}(z)$. By Lemma \ref{convlemma}, for each $\gs\in\Gal(K/\B{Q})$ we have $\limsup_{n\to\infty}|f^\gs(n)|_p^{1/n}\leq 1.$ Consequently, $\limsup_{n\to\infty}|f_i(n)|_p^{1/n}\leq 1$. By Lemma \ref{pfat}, the $|f_i(n)|_p$ are uniformly bounded. 

As there are only finitely many primes dividing $A$, we see there is a natural number  $N$ such that $N\cdot f_i(n)\in\B{Z}$ for each $i$ and every $n$. Note that $f_i(n)$ is a $K$--linear combination of the $f^\gs(n)$, each of which has the property that $\limsup_{n\to\infty}|f^\gs(n)|_\infty^{1/n}\leq 1$ by Lemma \ref{convlemma}, and hence $F_i(z)$ converges in the open unit disk of $\B{C}$. 

But $N\cdot F_i(z)$ is an algebraic power series with integer coefficients that converges in the open unit disk of $\B{C}$; hence by Theorem \ref{fatou}, it is a rational function. Since $F(z)=\sum_{i=1}^r \ga_i F_i(z)$, it too is rational. The result now follows from Theorem \ref{ratC}.
\end{proof}

\section{The algebraic case}\label{sec5}

In this section, we prove Theorem \ref{main}. The proof involves two more reductions. The first is an application of the Lefschetz principle showing that it is sufficient to prove our result in the case that $K$ is a finitely generated extension of $\B{Q}$; this is done in Lemma \ref{lefschetz}. The second reduction is obtained via Lemma \ref{bound}, which provides bounds that we use in the reduction from the case that $K$ is finitely generated over $\B{Q}$ to the case that $K$ is a finite extension of $\B{Q}$.

\begin{lemma}\label{lefschetz} Let $K$ be a field of characteristic $0$ and $F(z)=\sum_{n\geq 1}f(n)z^n\in K[[z]]$ be a $D$--finite power series. Then there exists a finitely generated $\B{Q}$--subalgebra $R$ of $K$ and a nonzero polynomial $p(x)\in R[x]$ with $p(1),p(2),\ldots$ all nonzero such that $f(n)\in S^{-1}R$ for all $n$, where $S$ is the multiplicatively closed subset of $R$ generated by $p(1),p(2),\ldots$. In the case that $F(z)$ is the power series expansion of an algebraic function over $K(z)$, there exists a finitely generated $\B{Q}$--subalgebra $T$ of $K$ such that $f(n)\in T$ for all $n$.
\end{lemma}

\begin{proof} Since $F(z)$ is $D$-finite, there exist polynomials $P_0(z),\ldots ,P_r(z)\in K[z]$ with $P_0(z)\neq 0$ such that
$$P_0(n)f(n)+P_1(n)f(n-1)+\cdots +P_r(n)f(n-r)=0$$ for $n\geq r$ \cite[Proposition 6.4.3]{Stan1}.  As $P_0(z)\neq 0$, there is a natural number $N>r$ such that $P_0(n)\neq 0$ for $n> N$. Let $R\subseteq K$ denote the finitely generated $\B{Q}$--subalgebra of $K$ generated by the coefficients of $P_0(z),\ldots ,P_r(z)$ and $f(1),\ldots,f(N)$.

Let $p(x)=P_0(x+N)$ and let $S$ be the multiplicatively closed subset of $R$ generated by $p(1),p(2),\ldots$. We claim that $f(n)\in S^{-1}R$ for every natural number $n$. If not, then there is a minimal natural number $n$ such that  $f(n)\notin S^{-1}R$. By construction $n$ is greater than $N$. Thus \begin{align*}f(n)&=-P_0(n)^{-1}(P_1(n)f(n-1)+\cdots +P_r(n)f(n-r))\\
&=-p(n-N)^{-1}(P_1(n)f(n-1)+\cdots +P_r(n)f(n-r)),\end{align*} which is in $S^{-1}R$, a contradiction. Hence the claim is valid.

We now consider the case that $F(z)$ is algebraic over $K(z)$. Note that $F(z)$ is $D$-finite \cite[Theorem 6.4.6]{Stan1}. It follows that there is a finitely generated $\B{Q}$--subalgebra $R$ of $K$ such that $f(n)$ is in the field of fractions of $R$ for every natural number $n$. Then by Lemma \ref{geneis}, there is some nonzero $A\in R$ such that $A^nf(n)\in R$ for all natural numbers $n$. We now take $T$ to be the finitely generated $\B{Q}$--algebra obtained by adjoining $1/a$ to $R$.  Then $f(n)\in T$ for all natural numbers $n$.
\end{proof}

\begin{lemma}\label{bound} Let $K$ be a finite extension of degree $d$ over $\B{Q}$ and $f:\B{N}\to K$ be a multiplicative function. Suppose that $F(z)=\sum_{n\geq 1}f(n)z^n\in K[[z]]$ is the power series expansion of a rational function in $K(z)$. If $F(z)$ satisfies a polynomial $P(z,F(z))=0$, where $P(z,y)=\sum_{i=0}^rP_i(y)z^i\in K[z,y]$, then $F(z)=A(z)/B(z),$ where $A(z)\in K[z]$ has $$\deg A(z)\leq \max_{0\leq i\leq r-1}\big(\deg P_i(z)\big)+(rd!-1)\cdot\deg P_r(z),$$ $B(z)\in\B{Q}[z]$ has degree at most $d!\cdot \deg P_r(z)$, and all roots of $B(z)$ are roots of unity.
\end{lemma}

\begin{proof} Note that $G(z):=P_r(z)F(z)$ satisfies a polynomial equation of the form $$G(z)^r+Q_{r-1}(z)G(z)^{r-1}+\cdots+Q_0(z)=0.$$ Thus $G(z)$ is rational function in $K(z)$ that is integral over $K[z]$. Since $K[z]$ is integrally closed, $G(z)$ is a polynomial. Hence $F(z)$ can be written as $$F(z)=\frac{C(z)}{P_r(z)},$$ for some polynomial $C(z)\in K[z]$. Now by Theorem \ref{ratC}, all poles of $F(z)$ are roots of unity; thus we can write $$F(z)=\frac{A_0(z)}{B_0(z)},$$ where $A_0(z)\in K[z]$, $B_0(z)\in K[z]$ divides $P_r(z)$, and the zeros of $B_0(z)$ are all roots of unity.  Note that $K$ is contained in a Galois extension $L$ of $K$ with $[L:\B{Q}]\leq d!$.  Taking $$B(z) = \prod_{\gs\in {\rm Gal}(L/\B{Q})} B_0^{\gs}(z),$$
we see we can
write $$F(z)=\frac{A(z)}{B(z)},$$ where $A(z)\in K[z]$ and $B(z)\in\B{Q}[z]$ is of degree at most $d!\cdot \deg P_r(z)$.

To finish the proof, it only remains to prove the bound on $\deg A(z)$. Note that we have $$P_r(z)A(z)^r=-P_{r-1}(z)A(z)^{r-1}B(z)-\cdots-P_0(z)B(z)^r.$$ Comparing degrees we see that \begin{align*}\deg P_r(z)+r\cdot\deg A(z) &\leq \max_{0\leq i\leq r-1}\big(\deg P_i(z)+i\cdot\deg A(z)+(r-i)\cdot\deg B(z)\big)\\
&\leq \max_{0\leq i\leq r-1}\big(\deg P_i(z)+i\cdot\deg A(z)+(r-i)d!\cdot\deg P_r(z)\big)\\
&\leq \max_{0\leq i\leq r-1}\big(\deg P_i(z)\big)+(r-1)\deg A(z)+rd!\cdot\deg P_r(z),
\end{align*} and so we have the bound \begin{equation*} \deg A(z)\leq \max_{0\leq i\leq r-1}\big(\deg P_i(z)\big)+(rd!-1)\cdot\deg P_r(z).\qedhere\end{equation*}
\end{proof}

\begin{proof}[Proof of Theorem \ref{main}] By Lemma \ref{lefschetz}, there is a finitely generated $\B{Q}$--subalgebra $R$ of $K$ such that $f(n)\in R$ for all $n$. 
Then $F(z)$ satisfies a polynomial equation
$$\sum_{i=0}^r P_i(z)F(z)^i \ = \ 0$$ for some polynomials 
$P_0(z),\ldots ,P_r(z)\in R[z]$.

By Noether normalization \cite[Theorem 13.3]{Eis1}, there exist $t_1,\ldots,t_e\in R$ such that $t_1,\ldots, t_e$ are algebraically independent over $\B{Q}$, and $R$ is a finite $\B{Q}[t_1,\ldots,t_e]$--module.

Let $\C{S}\subseteq \spec(\B{Q}[t_1,\ldots,t_e])$ be the set of maximal ideals of the form $$(t_1-\ga_1,t_2-\ga_2,\ldots,t_e-\ga_e)$$ with $(\ga_1,\ldots,\ga_e)\in\B{Q}^e$. We write $$R=\sum_{i=1}^d \B{Q}[t_1,\ldots,t_e]u_i$$ with $u_i\in R.$ 

Since $R$ is a finite $\B{Q}[t_1,\ldots ,t_e]$--module, it is integral over $\B{Q}[t_1,\ldots ,t_e]$ \cite[Corollary 5.4, p.~396]{Hun1}.  
Consequently, by the Going--Up theorem \cite[Proposition 4.15]{Eis1}, for every $\F{m}\in\C{S}$ there is a maximal ideal $\widetilde{\F{m}}$ of $R$ with $$\widetilde{\F{m}}\cap \B{Q}[t_1,\ldots,t_e]=\F{m}.$$ 

Note that $R/\widetilde{\F{m}}$ is spanned by the images of $u_1,\ldots,u_d$ as a $\B{Q}$--vector space, since $\B{Q}[t_1,\ldots,t_e]/\F{m} \cong \B{Q}$. Hence $$\left[R/\widetilde{\F{m}}:\B{Q}\right]\leq d.$$  Thus $(f(n)+\widetilde{\F{m}})$ is a sequence in a finite extension $K_0$ of $\B{Q}$ with $[K_0:\B{Q}]\leq d$. Moreover, it is a multiplicative sequence and the power series $$y:=\sum_{n\ge 1} (f(n)+\widetilde{\F{m}})z^n,$$ satisfies the image of of the polynomial equation $P(y,z)=0$ mod $\widetilde{\F{m}}$. Thus by Lemma \ref{bound}, $y$ is a rational function $A(z)/B(z)$ with $A(z)\in (R/\widetilde{\F{m}})[z]$ and $B(z)\in \B{Q}[z]$ has all its zeros at roots of unity, and the degrees of $A(z)$ and $B(z)$ are bounded by a function of $d$ and the degrees of $P_0(z),\ldots, P_r(z)$.  

Since there are only finitely many monic polynomials in $\B{Q}[z]$ of fixed degree whose zeros are all roots of unity, we see that there is a single polynomial $C(z)\in\B{Q}[z]$ such that every rational polynomial up to some fixed degree whose zeros are all roots of unity divides $C(z)$. In particular, there exists a polynomial $C(z)\in\B{Q}[z]$ such that for every maximal ideal $\widetilde{\F{m}}$ lying over a maximal ideal $\F{m}\in \C{S}$, we have $$C(z)\sum_{n\geq 1}\left(f(n)+\widetilde{\F{m}}\right)z^n\in (R/\widetilde{\F{m}})[z],$$ and there is a uniform bound on the degrees of these polynomials as $\F{m}$ ranges over $\C{S}$.

Now consider $C(z)F(z)$. By the above remarks, there is a natural number $N$ such that for every maximal ideal $\widetilde{\F{m}}$ lying over a maximal ideal $\F{m}\in \C{S}$ and $n>N$ the coefficient of $z^n$ in $C(z)F(z)$ is in $\widetilde{\F{m}}$.

But $$\bigcap_{\F{m}\in\C{S}}\F{m}=(0),$$ and so $$\bigcap_{\substack{\widetilde{\F{m}}\ {\rm lying\ over~}\\ \F{m}\in S}}\widetilde{\F{m}}=(0).$$ Thus the coefficient of $z^n$ in $C(z)F(z)$ is $0$ for all $n>N$; that is, $C(z)F(z)$ is a polynomial. Thus $F(z)$ is a rational function, hence the result is given by Theorem \ref{ratC}.
\end{proof}

\section{The complete $D$--finite case}\label{Df}

In this section we prove Theorem \ref{mainD}. To achieve this we use a mixture of B\'ezivin's methods \cite{Bez1} as well as our own methods in order to remove the nonzero condition from the $D$--finite case.

\begin{lemma}[B\'ezivin \cite{Bez1}]\label{Bez3.5} Suppose that $F(z)=\sum_{n\geq 1}f(n)z^n\in\B{C}[[z]]$ is a $D$--finite power series with multiplicative coefficients. Then all of the singularities of $F(z)$ are located at roots of unity.
\end{lemma}

\begin{remark} Lemma \ref{Bez3.5} is stated above with only the necessary conditions for B\'ezivin's proof. This is Lemma 3.5 in \cite{Bez1}. For those interested in more details, see Appendix \ref{Apx1} for B\'ezivin's proof.
\end{remark}

\begin{lemma}\label{fred} Suppose $G(z):=\sum_{n\geq 0}g(n)z^n\in\B{C}[[z]]$ has only finitely many singularities and they are all located at roots of unity. Then there is a natural number $N$ such that for $0\leq j\leq N$ the only possible singularity of the function $G_j(z):=\sum_{n\geq 0} g(Nn+j)z^n$ occurs at $z=1$.
\end{lemma}

\begin{proof} Since all of the singularities of $G(z)$ occur at roots of unity, there is a natural number $N$ such that if $z=\zeta$ is a singularity of $G(z)$ we have $\zeta^N=1$. Note that \begin{equation}\label{Gj}z^jG_j(z^N)=\sum_{k=0}^{N-1}G(e^{2\pi i k/N}z)e^{-2\pi i kj/N}.\end{equation} Now all of the singularities of the right--hand side of \eqref{Gj} are located at $N$--th roots of unity, and hence so are all of the singularities of $G_j(z^N)$. It follows that if $G_j(z)$ has a singularity it is located at $z=1$.
\end{proof}

\begin{lemma}[Agmon \cite{Agm1}]\label{agmon} Let $A(z):=\sum_{n\geq 1}a(n)z^n\in\B{C}[[z]]$ be a power series whose coefficients are all nonzero. Denote $B(z):=\sum_{n\geq 1}z^n/a(n)\in\B{C}[[z]]$. Suppose that the $\lim_{n\to\infty}\sqrt[n]{|a(n)|}=1$ and that $A(z)$ has only $z=1$ as a singularity on the unit circle. Then either $B(z)$ admits the unit circle as a natural boundary or $B(z)$ has only $z=1$ as a singularity on the unit circle and $a(n+1)/a(n)\to 1$ as $n\to\infty$.
\end{lemma}

\begin{proof} This is stated as Lemma 2.9 in \cite{Bez1}; its proof is given by Agmon \cite{Agm1}.
\end{proof}

\begin{proposition}\label{ghlemma} Let $f(n)$ be a complex--valued sequence, let $\ga\in\B{C}$, and let $M\in\B{N}$. Define \begin{enumerate}
\item[(i)] $g(n)= f(n)-f(n-M)$ for $n>M$
\item[(ii)] $h(n)=g(n)-\ga g(n-M)$ for $n>2M$.\end{enumerate} If $|f(n)|=O(1)$ and $|h(n)|=O(n^{-\eps})$ for some $\eps\in(0,1]$ then there is a $M_1\in\B{N}$ such that $|f(n)-f(n-M_1)|=O(n^{-\eps/2}).$
\end{proposition}

\begin{proof} We start by making a few reductions. Firstly, by considering the $M$ sequences $$f_i(n):=f(Mn+i)$$ for $0\leq i<M$, we see that we may assume $M=1$. Next, by considering the real and imaginary parts of $f(n)$ separately, we may assume that $f(n)$ is real--valued. Thirdly, if $f(n)$ is real--valued and $M=1$, we may assume that $\ga\in\B{R}$. To see this, note that $$\Im(h(n))=\Im(\ga)g(n-1).$$ Thus if $\ga\notin\B{R}$, $$|f(n-1)-f(n-2)|=|g(n-1)|=\frac{\Im(h(n))}{\Im(\ga)}=O(n^{-\eps}).$$ We thus assume that $M=1$, $f:\B{N}\to\B{R}$, and $\ga\in\B{R}$.

We divide the remainder of the proof into four cases. We make use of the identity \begin{equation}\label{*} g(N)-\ga^{r+1}g(N-r-1)=h(N)+\ga h(N-1)+\ga^2 h(N-2)+\cdots+\ga^rh(N-r)\end{equation} for $N>r\geq 0$ and $N-r\geq 3$. In particular, since there is a $C>0$ such that $|h(n)|<Cn^{-\eps}$ for all $n\geq 3$, we have \begin{equation}\label{**} |g(N)-\ga^{r+1}g(N-r-1)|\leq \sum_{i=0}^r\frac{C|\ga|^i}{(N-i)^\eps}<\frac{C(r+1)}{(N-r)^\eps}\max\{1,|\ga|^r\}.\end{equation}

{\sc Case I:} $|\ga|>1$. In this case \eqref{**} gives $$|\ga|^{r+1}|g(N-r-1)|<\frac{C(r+1)}{(N-r)^\eps}|\ga|^r+|g(N)|.$$ We take $r=\left\lfloor \frac{2\log n}{\log |\ga|}\right\rfloor$ and $N=n+r+1$. Then $$|g(n)|<\frac{C\left(\left\lfloor \frac{2\log n}{\log |\ga|}\right\rfloor+1\right)}{(n+1)^\eps|\ga|}+\frac{|g(n+r+1)|}{|\ga|^{r+1}}.$$ Note that $$\frac{C\left(\left\lfloor \frac{2\log n}{\log |\ga|}\right\rfloor+1\right)}{(n+1)^\eps|\ga|}=O\left(\frac{\log n}{(n+1)^\eps}\right).$$ Also $$|g(n+r+1)|=|f(n+r+1)-f(n+r)|\leq |f(n+r+1)|+|f(n+r)|=O(1)$$ and $$|\ga|^{r+1}>|\ga|^{\frac{2\log n}{\log |\ga|}}=n^2.$$ So $$\frac{g(n+r+1)}{|\ga|^{r+1}}=O\left(\frac{1}{n^2}\right).$$  It follows that $|g(n)|=O(n^{-\eps/2})$ in this case, so we take $M_1=1.$

{\sc Case II:} $|\ga|<1$. Equation \eqref{**} gives $$|g(N)|<\frac{C(r+1)}{(N-r)^{\eps}}+|\ga|^{r+1}|g(N-r-1)|.$$ If $\ga=0$ we are done, so assume that $\ga\neq 0$. We take $N=n$ and $r=\left\lfloor \frac{2\log n}{-\log |\ga|}\right\rfloor.$ Then we see $$|g(n)|<\frac{C\left(\left\lfloor \frac{2\log n}{-\log |\ga|}\right\rfloor+1\right)}{\left(n-\left\lfloor \frac{2\log n}{-\log |\ga|}\right\rfloor\right)^\eps} +|\ga|^{r+1}|g(n-r-1)|.$$ Just as in the first case, we see that $$|g(n)|=O\left(\frac{\log n}{n^\eps}\right)$$ and hence taking $M_1=1$ gives the claim.

{\sc Case III:} $\ga=-1$. In this case \begin{align*} h(n) &= g(n)+g(n-1)\\
&= f(n)-f(n-1)+f(n-1)-f(n-2)\\
&=f(n)-f(n-2).\end{align*} Thus taking $M_1=2,$ we see that $$|f(n)-f(n-M_1)|=|h(n)|=O(n^{-\eps/2}).$$

{\sc Case IV:} $\ga=1$. For this case, we follow the proof of Lemma 1 in \cite{Bel1}. Equation \eqref{**} gives $$|g(n)-g(n-r-1)|<\frac{C(r+1)}{(n-r)^\eps}$$ for $0\leq r<n-2$. Notice that \begin{align*} |f(n)-f(n-t-1)|&= |g(n)+g(n-1)+\cdots+g(n-t)|\\
&=|(t+1)g(n)+(g(n-1)-g(n-2))+\\
&\quad\qquad\cdots+(g(n-t)-g(n))|\\
&>(t+1)|g(n)|-\sum_{r=0}^{t-1}|g(n)-g(n-r-1)|\\
&>(t+1)|g(n)|-\sum_{r=0}^{t-1}\frac{C(r+1)}{(n-r)^\eps}\\
&>(t+1)|g(n)|-\frac{C}{(n-t)^\eps}\cdot\frac{(t+1)t}{2}.
\end{align*} Since $|f(n)|=O(1)$ there is a $C_1>0$ such that $|f(n)-f(n-t-1)|<C_1$ for all $n$ and $t$. Thus $$(t+1)|g(n)|<C_1+\frac{C}{(n-t)^\eps}\cdot\frac{(t+1)t}{2}.$$ We take $t=\lfloor n^{\eps/2}\rfloor.$ Then $$n^{\eps/2}|g(n)|\leq (t+1)|g(n)|<C_1+\frac{C}{(n-t)^\eps}\cdot\frac{(t+1)t}{2}=O(1).$$ In particular, $|g(n)|=O(n^{-\eps/2}),$ and so taking $M_1=1$, we obtain the result. This completes the proof.
\end{proof}

\begin{proposition}\label{neps} Let $F(z):=\sum_{n\geq 1}f(n)z^n\in\B{C}[[z]]$ be a $D$--finite power series. Suppose that $f(n)=O(1)$ and there exists a $M\in \B{N}$ such that $|f(n)-f(n-M)|\to 0$. Then there exists $M_1\in\B{N}$ and an $\eps>0$ such that $$|f(n)-f(n-M_1)|=O(n^{-\eps}).$$
\end{proposition}

\begin{proof} Since $(f(n))_n$ is $P$--recursive, there exist polynomials $P_0,\ldots,P_d,$ not all zero, such that \begin{equation}\label{P*} P_0(n)f(n)+P_1(n)f(n-1)+\cdots+P_d(n)f(n-d)=0\end{equation} for all $n$ sufficiently large. 

Let $D=\max\{\deg P_0,\ldots,\deg P_d\}$ and let $c_i$ be the coefficient of $z^D$ in $P_i(z)$ (possibly zero). Then dividing \eqref{P*} by $n^D$ and using the fact that $|f(n)|=O(1)$, we see that $$\sum_{i=0}^d c_if(n-i)=O\left(\frac{1}{n}\right).$$ Let $Q(z)=\sum_{i=0}^dc_iz^i.$ By shifting indices if necessary, we may assume $c_0\neq 0$; moreover, we can take $c_0=1$. Factor $$Q(z)=(1-\gb_1z)\cdots(1-\gb_dz)$$ where some of the $\beta_i$ may be zero. We take $f_0(n)=f(n)$ and define $$f_i(n)=f_{i-1}(n)-\gb_if_{i-1}(n-1)$$ for $1\leq i\leq d.$ By construction $$f_d(n)=\sum_{i=0}^d c_i f(n-i)=O\left(\frac{1}{n}\right)$$ for $n>d$. Also, for all $i$, $|f_i(n)|=O(1)$. Note that $|f_d(n)-f_d(n-1)|=O(n^{-1})$. Pick $i$ minimal for which there exists a $M$ such that $|f_i(n)-f_i(n-M)|=O(n^{-\eps})$ for some $\eps\in(0,1]$. Hence we have $i\leq d$. If $i=0$, we are done, so we may assume $i>0$.

Let $g(n)=f_{i-1}(n)-f_{i-1}(n-M)$ and $h(n)=g(n)-\gb_i^M g(n-M)$. Notice that \begin{align*} |h(n)|&= |f_{i-1}(n)-f_{i-1}(n-M)-\gb_i^Mf_{i-1}(n-M)+\gb_i^Mf_{i-1}(n-2M)|\\
&=\left|\sum_{j=0}^{M-1}\gb_i^j(f_{i-1}(n-j)-\gb_i f_{i-1}(n-j-1))\right.\\
&\qquad\quad-\left.\sum_{j=0}^{M-1} \gb_i^j(f_{i-1}(n-M-j)-\gb_i f_{i-1}(n-M-j-1))\right|\\
&=\left|\sum_{j=0}^{M-1}\left(\gb_i^jf_{i}(n-j)+\gb_i^jf_i(n-M-j)\right)\right|\\
&\leq \sum_{j=0}^{M-1}|\gb_i^j|\cdot|f_i(n-j)-f_i(n-j-M)|\\
&=O(n^{-\eps}).\end{align*} Since $|f_{i-1}(n)|=O(1)$, we see by Proposition \ref{ghlemma} that there exists a $M_1$ such that $$|f_{i-1}(n)-f_{i-1}(n-M_1)|=O(n^{-\eps/2}).$$ This contradicts the minimality of $i$. Thus $i=0$ and the result follows.
\end{proof}

Within the proof of Theorem \ref{mainD}, we will use the following theorem of Indlekofer and K\'atai \cite{Kat1}, as well as a simple lemma on periodic multiplicative functions.

\begin{theorem}[Indlekofer and K\'atai \cite{Kat1}]\label{katai} Let $f:\B{N}\to\B{C}$ be a multiplicative function, and assume that $$\sum_{n\leq x}|f(n+M)-f(n)|=O(x)$$ for a suitable $M\in\B{N}$. Then either \begin{enumerate} \item[(a)] $\sum_{n\leq x}|f(n)|=O(x),$ or 
\item[(b)] $f(n)=n^{\gs+i\tau}u(n),$ $\gs\leq 1$, and $u$ is a complex--valued multiplicative function satisfying $$u(n+M)=u(n)\ \ (\forall n\geq 1),\quad u(n)=\chi_M(n)\ \ {\textnormal{\emph{if}}}\ \ (n,M)=1,$$ where $\chi_M(n)$ is a suitable multiplicative character $\bmod\ M$.\end{enumerate}
\end{theorem}

\begin{lemma}\label{p=1} Suppose that $\omega:\B{N}\to\B{C}$ is a periodic multiplicative function and let $p$ be a prime not dividing the period of $\omega(n)$. If $\omega(n)$ is not identically zero, then $|\omega(p)|=1$.
\end{lemma}

\begin{proof} Denote the period of $\omega(n)$ by $M$. We will first show that $\omega(n)$ is completely multiplicative on integers coprime to $M$. To this end, let $a$ and $b$ be integers each coprime to $M$; note that also, $ab$ is coprime to $M$. We have two cases: $a\equiv b\ (\bmod\ M)$ or $a\not\equiv b\ (\bmod\ M)$.

If $a\equiv b\ (\bmod\ M)$, then $b\equiv a+M\ (\bmod\ M)$. Since $\gcd(a,M)=1$ there are $x,y\in\B{Z}$ so that $ax+My=1$. But also, $$1=ax+My=ax+My-ay+ay=a(x-y)+(a+M)y.$$ Since $1$ can be written as a linear combination of $a$ and $a+M$, we have $\gcd(a,a+M)=1$. Thus using the $M$--periodicity of $\omega(n)$, we have $$\omega(ab)=\omega(a(a+M))=\omega(a)\omega(a+M)=\omega(a)\omega(b).$$

Now suppose that $a\not\equiv b\ (\bmod\ M)$. Since $\gcd(b,M)=1$ and $\{0,1,\ldots,a-1\}$ is a complete set of residues modulo $a$, we have that $$b,M+b,2M+b,\ldots,(a-1)M+b$$ is a complete set of residues modulo $a$. Thus there is a $k\in\{0,1,\ldots,a-1\}$ such that $\gcd(kM+b,a)=1$. As before, using the $M$--periodicity of $\omega(n)$, we have that $$\omega(ab)=\omega(a(kM+b))=\omega(a)\omega(kM+b)=\omega(a)\omega(b).$$ Thus we have shown that $\omega(n)$ is completely multiplicative on integers corpime to $M$.

To finish our proof, suppose that $\omega(n)$ is not identically zero and let $p$ be a prime not dividing $M$. By a classical theorem of Euler, we have $p^{\varphi(M)}\equiv 1\ (\bmod\ M).$ Using the completely multiplicative property of $\omega(n)$ on integers coprime to $M$, we have $$\omega(1)=\omega(p^{\varphi(M)})=\omega(p)^{\varphi(M)}.$$ Since $\omega(n)$ is not identically zero, we have $\omega(1)=1$. The above equation gives $|\omega(p)|^{\varphi(M)}=1$ so that $|\omega(p)|=1$.
\end{proof}

\begin{proof}[Proof of Theorem \ref{mainD}] Let $f:\B{N}\to\B{C}$ be a multiplicative function with $D$--finite generating series. Define $h=f\overline{f}:\B{N}\to\B{R}.$ Since $h$ is real--valued we may apply B\'ezivin's result for real functions (Theorem \ref{BezivinDF}) to give $h(n)=n^k\omega(n)$. By Lemma \ref{p=1} since $h$ is not eventually zero, $|\omega(p)|=1$ for all primes $p$ larger than the period of $\omega(n)$. Thus $|h(p)|=p^k$ for $p$ any prime large enough. There are two cases: either $k$ is even, or $k$ is odd.

In the case that $k$ is even, denote $$f_0(n):= \begin{cases} \frac{f(n)}{n^{k/2}\sqrt{\omega(n)}} & {\rm if\ } \omega(n)\neq 0\\ 0 & {\rm if\ } \omega(n)=0.\end{cases}$$ Since $n^{k/2}$ is a polynomial in $n$, $f_0(n)$ has a $D$--finite generating series. Note that $f_0(n)$ is still multiplicative. We denote the generating series for $f_0(n)$ by $F_0(z)$. By Lemma \ref{Bez3.5}, all of the singularities of $F_0(z)$ are located at roots of unity. Now define $$g(n):=\begin{cases} f_0(n) & {\rm if\ } f_0(n)\neq 0\\ 1 & {\rm if\ } f_0(n)=0.\end{cases}$$ Denote the generating series of $g(n)$ by $G(z)$. Note that $\{n:f_0(n)=0\}=\{n:\omega(n)=0\}$, which is a finite union of complete arithmetic progressions.  Thus $G(z)-F_0(z)$ is a rational function whose poles are located at roots of unity.  Since the $F_0(z)$ is $D$--finite and its coefficients are multiplicative, its singularities are located at roots of unity by Lemma 6.2.  Thus the singularities of $G(z)$ are located at roots of unity. By Lemma \ref{fred} there is a natural number $N$ such that for $0\leq j\leq N$ the function $G_j(z):=\sum_{n\geq 0} g(Nn+j)z^n$ has only $z=1$ as a singularity or $G_j(z)$ is an entire function. Since $|g(n)|=1$ for all $n$, the functions $G_j(z)$ are not entire. By construction, each $G_j(z)$ is $D$--finite and $$\overline{G_j}(z)=\sum_{n\geq 1}\frac{z^n}{g(Nn+j)}$$ since $|g(n)|=1$. The function $\overline{G_j}(z)$ is also $D$--finite and hence has only finitely many singularities. Applying Lemma \ref{agmon}, we see that for $0\leq j<N$ we have $$\lim_{n\to\infty}g(N(n+1)+j)-g(Nn+j)=0.$$ Since this holds for all $0\leq j<N$, we have $$\lim_{n\to\infty}g(n+N)-g(n)=0.$$

Denote the period of $\omega(n)$ by $M$. Then $f_0(n+NM)-f_0(n)\to 0$ since either $f_0(n+NM)=f_0(n)=0$ or $$f_0(n+NM)-f_0(n)=g(n+NM)-g(n)\to 0.$$ In either case $$\lim_{n\to\infty}f_0(n+NM)-f_0(n)=0.$$ By Proposition \ref{neps}, there exists an $\eps>0$ and a $M_1\in\B{N}$ such that $$|f_0(n+M_1)-f_0(n)|=O(n^{-\eps}).$$ Denote $$\widetilde{f_0}(n):=f_0(n)n^{\eps}.$$ Then \begin{equation}\label{katai1}\sum_{n\leq x}\left|\widetilde{f_0}(n)\right|=\sum_{\substack{n\leq x\\ \omega(n)\neq 0}}n^\eps\sim Cx^{1+\eps},\end{equation} for some positive constant $C$, since the set of $n$ for which $\omega(n)\neq 0$ is a nonempty finite union of arithmetic progressions. Equation \eqref{katai1} is the first step towards applying Theorem \ref{katai}. For the next step, notice that \begin{align*} \sum_{n\leq x}|\widetilde{f_0}(n+M_1)-\widetilde{f_0}(n)| &=\sum_{n\leq x}|f_0(n+M_1)(n+M_1)^\eps-f_0(n)n^\eps)|\\
&\leq \sum_{n\leq x}\big[|f_0(n+M_1)|((n+M_1)^\eps-n^\eps)\\
&\qquad +n^\eps|f_0(n+M_1)-f_0(n)|\big]\\
&=\sum_{n\leq x} O(1)=O(x).
\end{align*} By Theorem \ref{katai}, we have $\widetilde{f_0}(n)=n^{\gs+i\tau}u(n)$ where $\gs\leq 1$ and $u(n)$ is a multiplicative periodic function. Hence $$f_0(n)=n^{\gs-\eps+i\tau}u(n).$$ Since $|f_0(n)|=1$ for infinitely many $n$, we see that $\gs-\eps=0$, and so $$f_0(n)=n^{i\tau}u(n).$$ Since $u(n)$ is periodic and not identically zero, it must be the case that $\tau=0$ (cf. proof of Proposition 2 of \cite{Ger1}); thus $f_0(n)=u(n)$. 

Recall that $$f(n)=f_0(n)n^k\sqrt{\omega(n)}=n^ku(n)\sqrt{\omega(n)},$$ and the product $u(n)\sqrt{\omega(n)}$ is both periodic and multiplicative, which proves the theorem in the case that $k$ is even.

Now suppose that $k$ is odd. Let $f(n)$ be multiplicative and $P$--recursive and suppose that for all sufficiently large primes $p$ we have $|f(p)|^2 = p^k$. Then $$h(n):=f(n)^2$$ is multiplicative, $P$--recursive and has $|h(p)|^2=p^{2k}$ for $p$ large enough. By the result in the even case, $$h(n)=n^k \omega(n)$$ for some periodic multiplicative function $\omega(n)$. Pick $\omega_0(n)$ periodic with $\omega_0(n)^2=\omega(n)$. Then $$f(n)^2 = n^k \omega_0(n)^2,$$ so that $$f(n) = \eps_n n^{k/2} \omega_0(n)$$ where $\eps_n\in\{+1,-1\}$ for all $n$. Notice now that if we pick an arithmetic progression $am+b$ on which $\omega_0(n)$ is a nonzero constant $c$, then $$f(am+b) =(am+b)^{k/2}\eps_{am+b} c .$$ Let $K:=\B{Q}(f(1),f(2),f(3),\ldots)(c).$ Recall that $K$ is finitely generated by the $D$--finite hypothesis. But $$(am+b)^{k/2} = f(am+b)\eps_{am+b} c^{-1}\in K$$ for every $m\geq 1$, which contradicts that $K$ is a finitely generated extension of $\B{Q}$.
\end{proof}

\section{Concluding remarks}

Throughout our investigation, multiplicative periodic functions have played an essential role. We have chosen to denote these functions by $\chi$. This notation is not used coincidentally, but because of a relationship to Dirichlet characters. Indeed, let $f:\B{N}\to\B{C}$ be a periodic  multiplicative function. Since $f$ is multiplicative, $f(1)=1$. Now let $n\in\B{N}$ with $(n,d)=1$. Using Dirichlet's theorem for primes in arithmetic progressions, one has that there are infinitely many primes of the form $dx+n$. Now let $k\in\B{N}$ and choose $k$ distinct primes $dx_1+n,dx_2+n,\ldots,dx_k+n$. Since $f$ is multiplicative, and primes are coprime to each other, using the $d$--periodicity of $f$ we have $$f(n^k)=f\left(\prod_{i=1}^k(dx_i+n)\right)=\prod_{i=1}^kf(dx_i+n)=f(n)^k;$$ hence $f$ is completely multiplicative when restricted to the positive integers coprime to $d$. Also, by Euler's theorem, we have that $$f(1)=f(n^{\gf(d)})=f(n)^{\gf(d)},$$ so that $f(n)$ is a $\gf(d)$--th root of unity. Hence for all $n$ with $(n,d)=1$, $f$ agrees with a Dirichlet character $\chi$ modulo $d$. Indeed, one may describe these functions completely.	

\begin{theorem}[Leitmann and Wolke \cite{LW1}] Let $f:\B{N}\to\B{C}$ be a multiplicative function with period $d=p_1^{l_1}\cdots p_r^{l_r}$. Then for $i=1,\ldots,r$, there exist $n_i\in\B{N}$ ($0\leq n_i\leq l_i$) such that $$f(p^l)=\begin{cases} \chi(p^l) & (p,d)=1\\ a_{i,l} & p=p_i\ (i=1,\ldots,r)\end{cases}$$ where $a_{i,j}\in\B{C}$  ($i=1,\ldots,r$ and $j\geq 0$) with \begin{align*} a_{i,0} &=1,\\ a_{i,l} &=0 \qquad (n_i<l\leq l_i),\\ a_{i,l_i+t} &=a_{i,l_i}\chi^*(p_i^t) \qquad (i=1,\ldots,r),\end{align*} where $\chi$ is a character modulo $d$ and $\chi^*$ is a character modulo $d=p_1^{l_1-n_1}\cdots p_r^{l_r-n_r}.$
\end{theorem}

\noindent We note that the conclusion of this theorem holds over any field of characteristic zero, once again appealing to the Lefschetz principle.

In the case of positive characteristic, algebraic functions are much more pathological. For example, while Fatou's theorem shows that an algebraic function whose coefficients are uniformly bounded is rational, the function $$F(z)=\sum_{n\geq 0}z^{2^n}\in\B{F}_2[[z]]$$ is algebraic over $\B{F}_2(z)$, but is, nevertheless, irrational. Note that $F(z)$ is the generating function of a multiplicative function, and hence the conclusion of Theorem \ref{main} does not hold in positive characteristic. Christol gives a characterization of algebraic functions over finite fields in terms of automatic sequences.

\begin{theorem}[Christol \cite{Chr1}] Let $q=p^k$ be a prime power, let $\mathbb{F}_q$ be a finite field of size $q$, and let $(u_n)_{n\geq 0}$ a sequence with values in $\mathbb{F}_q$. Then, the sequence $(u_n)_{n\geq 0}$ is $p$--automatic if and only if the formal power series $\sum_{n\geq 0} u_n X^n$ is algebraic over $\mathbb{F}_q(X)$.
\end{theorem}

\noindent In light of Christol's theorem, it is natural to ask if one can characterize automatic multiplicative functions. Partial progress has been made by Yazdani \cite{Yaz1} and Coons \cite{Coons1}. All examples of automatic multiplicative functions found thus far have the property that they are well behaved on the set of prime powers. We make this more explicit in the following conjecture.

\begin{conjecture} Let $k\geq 2$ and $f$ be a  $k$--automatic multiplicative function. Then there is an eventually periodic function $g$ such that $f(p)=g(p)$ for every prime $p$.
\end{conjecture}

Finally, we note that in positive characteristic, Kedlaya \cite{Kedlaya} has pointed out that the algebraic closure of the Laurent power series over a field $K$ is not as well behaved as in the characteristic $0$ case.  To alleviate this difficulty, he looks at the algebra of \emph{Hahn power series}, $K((z^{\B{Q}}))$.  In this ring, we take all power series of the form
$$\sum_{\alpha\in \B{Q}} c_{\alpha} z^{\alpha}$$ with $c_{\alpha}\in K$ such that the set of $\alpha\in \B{Q}$ for which $c_{\alpha}\not = 0$ is well-ordered.  The advantage of working with this ring is that it is algebraically closed.  Kedlaya also extends the notion of being $k$-automatic, for a natural number $k$, to functions whose domain is the rational numbers.  Given a finite set $\Delta$, it would be interesting to characterize completely multiplicative maps 
$f:\mathbb{Q}\rightarrow \Delta$ that are $k$-automatic in the sense of Kedlaya.\\

\noindent\tbf{Acknowledgement.} We would like to thank Jean--Paul B\'ezivin for directing us to his paper and also for making many helpful comments.

\appendix

\section{Proof of Lemma \ref{Bez3.5}}\label{Apx1}

This appendix contains a B\'ezivin's proof of Lemma \ref{Bez3.5}. All of lemmas and their proofs are translated versions from B\'ezivin \cite{Bez1} with some corrected typos and slight modifications for ease of reading. They are added here for completeness (see the remark after Lemma \ref{Bez3.5}).

To give the proof of Lemma \ref{Bez3.5}, we will need the following lemmas from \cite{Bez2, Bez1} in their originally stated form.

\begin{lemma}[B\'ezivin \cite{Bez2}]\label{B2.2} Suppose that $\psi(z):=\sum_{n\geq 0}a(n)z^n\in\B{C}[[z]]$ is a $D$--finite power series. Let $$\sum_{j=0}^t P_j(n)a(n+j)=0$$ be the recurrence relation satisfied by $a(n)$, where $P_i(z)\in\B{C}[x]$ with $P_t(z)$ nonzero. Let $q$ be a positive integer. Then the function $b(n):=a(nq)$ satisfies a recurrence relation of the form $$\sum_{k=0}^m H_k(n)b(n+k)=0,$$ where $H_k(z)\in\B{C}[z]$, $H_m(z)$ nonzero, and $m\leq t.$ Moreover, if the radius of convergence of $\psi(z)$ is finite and nonzero, the singularities of the series $\sum_{n\geq 0}b(n)z^n\in\B{C}[[z]]$ are among the $q$--th powers of the singularities of the series $\psi(z)$.
\end{lemma}

\begin{lemma}[B\'ezivin \cite{Bez1}]\label{B3.1} Let $a,b:\B{N}\to\B{C}$ be two functions satisfying the recurrences $$\sum_{i=0}^t P_i(n)a(n+i)=0\qquad\mbox{and}\qquad \sum_{j=0}^s Q_j(n)b(n+j)=0$$ for all $n$, where $P_i(z),Q_j(z)\in\B{C}[z]$ with $P_t(z)$ and $Q_s(z)$ nonzero. Then $c(n):=a(n)+b(n)$ satisfies a relation of the same form $$\sum_{k=0}^m H_k(n)c(n+k)=0,$$ where $H_k(z)\in\B{C}[z]$, $H_m(z)$ is nonzero, and $m\leq s+t$, and also the function $d(n):=a(n)b(n)$ satisfies a relation of the same form $$\sum_{k=0}^r M_k(n)d(n+k)=0,$$ where $M_k(z)\in\B{C}[z]$, $M_r(z)$ is nonzero, and $r\leq st$.
\end{lemma}

Lemma \ref{B2.2} was originally given in \cite[p.~137]{Bez2}, and is also stated as Lemma 2.2 of \cite{Bez1}. Lemma \ref{B3.1} was stated as Lemma 3.1 of \cite{Bez1}.

We will also need the following Lemmas from \cite{Bez1} in slightly different forms from how they were originally stated. Lemma \ref{B3.2} and Lemma \ref{B3.3} were originally stated as Lemmas 3.2 and 3.3 of \cite{Bez1}.

\begin{lemma}[B\'ezivin \cite{Bez1}]\label{B3.2} Suppose that $f:\B{N}\to\B{C}$ is a multiplicative function that is $P$--recursive, satisfying the recurrence $$\sum_{i=0}^t P_i(n)f(n+i)=0,$$ where $P_i(z)\in\B{C}[z]$ and $P_t(z)$ is nonzero. Set $N=(2t+1)!$ and let $q$ be an integer coprime to $N$. Then we have $$f(nq)=f(n)f(q)$$ for all $n$ sufficiently large.
\end{lemma}

\begin{proof} Set $u(n)=f(nq)-f(n)f(q).$ By Lemma \ref{B2.2} and Lemma \ref{B3.1}, the function $u(n)$ satisfies a recurrence of the form $$\sum_{i=0}^m H_i(n)u(n+i)=0,$$ where $H_i(z)\in\B{C}[z]$, $H_m(z)$ is nonzero, and $m\leq 2t$.

Let $n$ be an integer of the form $n=k+hq$ with $1\leq k\leq 2t+1$ and $h\in\B{N}$. Then $n$ is coprime to $q$ by the above hypotheses.

Thus by the multiplicativity of the function $f(n)$, we have that $u(n)=0$ for all such integers $n$.

Now let $h$ be a large enough integer so that for $n\geq 1+hq$ we have $H_m(n)\neq 0$.  Thus we have for $m\leq 2t$ that $$u(1+hq)=\cdots=u(m+1+hq)=0.$$ The recurrence relation and the above hypothesis on $h$ thus implies $u(k+hq)=0$ for all $k\geq 1$, which proves the lemma.
\end{proof}

\begin{lemma}[B\'ezivin \cite{Bez1}]\label{B3.3} Suppose that $F(z)=\sum_{n\geq 1}f(n)z^n\in\B{C}[[z]]$ is a $D$--finite power series with multiplicative coefficients, and suppose that $f(n)$ is not eventually zero. Then there is a constant $P_0$ such that $f(p^k)\neq 0$ for all primes $p\geq P_0$ and all $k\in\B{N}$.
\end{lemma}

\begin{proof} Let $$\sum_{i=0}^t P_i(n)f(n+i)=0$$ be the recurrence relation satisfied by $f(n)$, where $P_i(z)\in\B{C}[z]$ with $P_t(z)$ nonzero. Towards a contradiction, suppose there exists an infinite set of prime powers $p_i^{k_i}$ such that $f(p_i^{k_1})=0$.

Let $N$ be a solution to the system of congruences $$N\equiv -i+p_i^{k_i}\ (\bmod\ p_i^{k_i+1})\ \ 1\leq i\leq t.$$

Let $h\in\B{N}$ and set $$M=N+hp_1^{k_1+1}\cdots p_t^{k_t+1}.$$ For all values of $i$ and each choice of $h$, the integer $M+i$ is, for all $i=1,\ldots,t$, divisible by $p_i^{k_i}$ but not by any power larger than $k_i$. From the multiplicativity of $f(n)$, we have that $f(M+i)=0$ for $i=1,\ldots,t$. 

If we choose $h$ large enough, we will have that $P_t(n)\neq 0$ for all $n\geq M$.

Utilizing the recurrence relation, we have that $f(n)$ is zero for all $n\geq M+1$. This proves the result.
\end{proof}

We are now in a position to give the proof of Lemma \ref{Bez3.5}. This is given as Lemma 3.5 in \cite{Bez1}.

\begin{proof}[Proof of Lemma \ref{Bez3.5}] Let $q$ be an integer satifying the conditions of Lemma \ref{B3.2}. We will suppose that $q$ is chosen so that $f(q)\neq 0$, which is possible in virtue of Lemma \ref{B3.3}.

By the equality $f(nq)=f(n)f(q)$ for large enough $n$, we have that $$\sum_{n\geq 1}f(nq)z^n=f(q)\sum_{n\geq 1}f(n)z^n+T(z)$$ where $T(z)\in\B{C}[z]$.

Let $\omega$ be a singularity of $g(z)=\sum_{n\geq 1}f(n)z^n$. By the preceding equality, $\omega$ is also a singularity of $\sum_{n\geq 1}f(nq)z^n$. By Lemma \ref{B2.2}, $\omega$ is a $q$--th power of a singularity $\omega'$ of $g(z)$. Because there are only finitely many singularities of $g(z)$, $\omega$ must be a root of unity.
\end{proof}


\bibliographystyle{amsplain}
\providecommand{\bysame}{\leavevmode\hbox to3em{\hrulefill}\thinspace}
\providecommand{\MR}{\relax\ifhmode\unskip\space\fi MR }
\providecommand{\MRhref}[2]{%
  \href{http://www.ams.org/mathscinet-getitem?mr=#1}{#2}
}
\providecommand{\href}[2]{#2}

\end{document}